\documentclass[3p,10pt]{elsarticle}

\usepackage{amssymb,latexsym,amsmath,color,subfigure,amsthm}
\journal{}

\newcommand{\set}[1]{\left\{#1\right\}}

\newcommand{\p}{\partial}
\newcommand{\mx}{\mathbf{x}}
\newcommand{\my}{\mathbf{y}}
\newcommand{\mz}{\mathbf{z}}

\newcommand{\mU}{\mathbf{U}}
\newcommand{\mV}{\mathbf{V}}
\newcommand{\mW}{\mathbf{W}}

\newcommand{\vt}{\boldsymbol{\theta}}
\newcommand{\vv}{\boldsymbol{\vartheta}}

\newtheorem{thm}{Theorem}[section]

\newtheorem{lem}[thm]{Lemma}

\begin{document}

\begin{frontmatter}



\title{Localization of small perfectly conducting cracks from far-field pattern with unknown frequency}


\author{Jung Ho Park}
\ead{jhoxxx@kookmin.ac.kr}
\author{Won-Kwang Park\corref{corpark}}
\ead{parkwk@kookmin.ac.kr}
\address{Department of Mathematics, Kookmin University, Seoul, 136-702, Korea.}
\cortext[corpark]{Corresponding author}

\begin{abstract}
In inverse scattering problem, it is well-known that subspace migration yields very accurate locations of small perfectly conducting cracks when applied frequency is known. In contrast, when applied frequency is unknown, inaccurate locations are identified via subspace migration with wrong frequency data. However, this fact has been examined through the experimental results so, the reason of such phenomenon has not been theoretically investigated. In this paper, we analyze mathematical structure of subspace migration with unknown frequency by establishing a relationship with Bessel function of order zero of the first kind. Identified structure of subspace migration and corresponding results of numerical simulation answer that why subspace migration with unknown frequency yields inaccurate location of cracks and gives an idea of improvement.
\end{abstract}


\end{frontmatter}





\section{Introduction}
It is well known that subspace migration is fast, effective and stable non-iterative detecting algorithm of small, perfectly conducting cracks in inverse scattering problem (see \cite{AGKPS,KP} for instance). However, for a successful application, information of applied frequency must be known. So, many researches assumed that applied frequency is known and investigated certain properties of single- and multi-frequency subspace migration algorithms, refer to \cite{AGKPS,JKHP,P1,P2,PP} and references therein.

However, if one has no \textit{a priori} information of applied frequency, subspace migration is inadequate to detect unknown targets. Particularly, in the problem of finding location of cracks, some information of location can be examined but identifying exact location is still impossible. Unfortunately, this fact has been examined heuristically through the results of numerical simulation so, as far as we know, mathematical analysis of subspace migration is still needed. This gives a motivation for this study to analyze structure of subspace migration and to develop an algorithm for finding exact location of cracks.

In this manuscript, we extend the research \cite{JKHP} of structure analysis of subspace migration with unknown frequency information. This is based on the fact that singular vectors associated with the nonzero singular values of Multi-Static Response (MSR) matrix can be represented by an asymptotic expansion formula in the existence of cracks. Throughout careful derivation, we identify that subspace migration imaging functional can be represented by the square of Bessel function of order zero of the first kind. Based on this representation, we investigate the reason why inexact locations of cracks are identified via subspace migration and develop a simple algorithm for finding exact locations.

This paper is organized as follows. In Section \ref{sec2}, we survey two-dimensional direct scattering problems, the asymptotic expansion formula in the presence of small perfectly conducting cracks, and subspace migration. In Section \ref{sec3}, we investigate the structure of subspace migration with unknown applied frequency by establishing a relationship with Bessel function of integer order of the first kind. Furthermore, we propose an exact location search algorithm by creating a small scatterer. Section \ref{sec4} presents some results of numerical simulations to support our investigation. Section \ref{sec5} presents a short conclusion.

\section{Preliminaries}\label{sec2}
In this section, we briefly introduce two-dimensional direct scattering problems in the presence of small, linear perfectly conducting crack(s), the asymptotic expansion formula, and subspace migration.

\subsection{Direct scattering problems and the asymptotic expansion formula}
First, we consider the two-dimensional electromagnetic scattering by $M-$different linear perfectly conducting cracks with same small length $2\ell$, denoted by $\Gamma_m$, $m=1,2,\cdots,M$, located in the homogeneous space $\mathbb{R}^2$ such that
\[\Gamma_m=\set{\mz_m=[x_m,y_m]^\mathrm{T}:-\ell\leq x_m\leq\ell},\]
and let $\Gamma$ be the collection of $\Gamma_m$. In this paper, we assume that $\Gamma_m$ are sufficiently separated from each other.

Let $u_{\mathrm{tot}}(\mx,\vt)$ satisfies following Helmholtz equation
\begin{equation}\label{Helmholtz}
  \left\{\begin{array}{rcl}
            \triangle u_{\mathrm{tot}}(\mx,\vt)+\omega^2u_{\mathrm{tot}}(\mx,\vt)=0 & \mbox{in} & \mathbb{R}^2\backslash\Gamma \\
            \noalign{\medskip}u_{\mathrm{tot}}(\mx,\vt)=0 & \mbox{on} & \Gamma,
          \end{array}
  \right.
\end{equation}
where $\omega=2\pi/\lambda$ denote \textit{unknown} angular frequency with wavelength $\lambda$ such that $\ell\ll\lambda$. In this paper, we assume that $\omega$ is positive definite and $\omega^2$ is not an eigenvalue of (\ref{Helmholtz}). Let us denote $u_{\mathrm{inc}}(\mx,\vt)=\exp(i\omega\vt\cdot\mx)$ be the incident plane wave with direction $\vt$ on the two-dimensional unit circle $\mathbb{S}^1$, and $u_{\mathrm{scat}}(\mx,\vt)$ be the unknown scattered field, which satisfies the Sommerfeld radiation condition
\[\lim_{|\mx|\to\infty}|\mx|^{1/2}\left(\frac{\p u_{\mathrm{scat}}(\mx,\vt)}{\p|\mx|}-iku_{\mathrm{ scat}}(\mx,\vt)\right)=0,\]
uniformly into all directions $\hat{\mx}=\mx/|\mx|$.

The far-field pattern $u_\infty(\hat{\mx},\vt;\omega)$ defined on $\mathbb{S}^1$ can be expressed in the form
\[u_{\mathrm{ scat}}(\mx,\vt)=\frac{e^{i\omega|\mx|}}{|\mx|^{1/2}}\left(u_\infty(\hat{\mx},\vt;\omega)+O\left(\frac{1}{|\mx|}\right)\right)\]
uniformly into all directions $\hat{\mx}=\mx/|\mx|$ and $|\mx|\longrightarrow+\infty$. Then, based on \cite{AKLP}, the far-field pattern can be represented as the following asymptotic expansion formula.

\begin{lem}[Asymptotic expansion formula] For $0<\ell<2$ and $\ell\ll\lambda$, the far-field pattern can be represented as follows:
  \begin{align}
  \begin{aligned}\label{AsymptoticFormula}
    u_\infty(\hat{\mx},\vt;\omega)&=-\frac{2\pi}{\ln(\ell/2)}\sum_{m=1}^{M}u_{\mathrm{inc}}(\mz,\vt;\omega)\overline{u_{\mathrm{ inc}}(\mz,\hat{\mx};\omega)}+O\bigg(\frac{1}{|\ln\ell|^2}\bigg)\\
    &=-\frac{2\pi}{\ln(\ell/2)}\sum_{m=1}^{M}\exp\bigg(i\omega(\vt-\hat{\mx})\cdot\mz_m\bigg)+O\bigg(\frac{1}{|\ln\ell|^2}\bigg).
  \end{aligned}
  \end{align}
\end{lem}

\subsection{Introduction to subspace migration}
At this point, we apply (\ref{AsymptoticFormula}) to explain an imaging technique known as subspace migration. From \cite{AGKPS}, subspace migration is based on the structure of singular vectors of the collected Multi-Static Response (MSR) matrix
\[\mathbb{K}=\bigg[u_\infty(\vv_j,\vt_l)\bigg]_{j,l=1}^{N}
=\left[
   \begin{array}{cccc}
     u_\infty(\vv_1,\vt_1) & u_\infty(\vv_1,\vt_2) & \cdots & u_\infty(\vv_1,\vt_N) \\
     u_\infty(\vv_2,\vt_1) & u_\infty(\vv_2,\vt_2) & \cdots & u_\infty(\vv_2,\vt_N) \\
     \vdots & \vdots & \ddots & \vdots \\
     u_\infty(\vv_N,\vt_1) & u_\infty(\vv_N,\vt_2) & \cdots & u_\infty(\vv_N,\vt_N)
   \end{array}
 \right],\]
where $u_\infty(\vv_j,\vt_l)$ is the far-field pattern with incident direction $\vt_l$ and observation direction $\vv_j$ for $j,l=1,2,\cdots,N$. For the sake of simplicity, we assume that we have coincide incident and observation directions, i.e., $\vv_j=-\vt_j$. Then, since the $jl-$th element of MSR matrix can be represented as
\[u_\infty(-\vt_j,\vt_l)=-\frac{2\pi}{\ln(\ell/2)}\sum_{m=1}^{M}\exp\bigg(i\omega(\vt_j+\vt_l)\cdot\mz_m\bigg),\]
$\mathbb{K}$ can be decomposed as follows
\begin{equation}\label{Decomposition}
  \mathbb{K}=-\frac{2\pi}{\ln(\ell/2)}\sum_{m=1}^{M}\mW(\mz_m;\omega)\mW(\mz_m;\omega)^T,
\end{equation}
where
\begin{equation}\label{VecW}
  \mW(\mx;\omega):=\frac{1}{\sqrt{N}}\bigg[\exp(i\omega\vt_1\cdot\mx),\exp(i\omega\vt_2\cdot\mx),\cdots,\exp(i\omega\vt_N\cdot\mx)\bigg]^T
\end{equation}

With this decomposition, we can introduce subspace migration imaging algorithm. Let us perform the Singular Value Decomposition (SVD) of $\mathbb{K}$ as
\begin{equation}\label{SVD}
  \mathbb{K}=\mathbb{UDV}^*\approx\sum_{m=1}^{M}\sigma_m\mU_m\mV_m^*.
\end{equation}
Then, by comparing (\ref{Decomposition}) and (\ref{SVD}), we can observe that the left- and right-singular vectors are satisfy
\begin{equation}\label{SingularVectors}
  \mU_m\approx\mathbf{W}(\mz_m;\omega)\quad\mbox{and}\quad\overline{\mV}_m\approx\mathbf{W}(\mz_m;\omega).
\end{equation}
Note that based on the orthonormal property of singular vectors, we can observe that
\begin{align}
\begin{aligned}\label{orthonormalproperty}
  &\langle\mathbf{W}(\mx;\omega),\mU_m\rangle\ne0\quad\mbox{and}\quad\langle\mathbf{W}(\mx;\omega),\overline{\mV}_m\rangle\ne0\quad\mbox{if}\quad\mx\approx\mz_m\\
  &\langle\mathbf{W}(\mx;\omega),\mU_m\rangle\approx0\quad\mbox{and}\quad\langle\mathbf{W}(\mx;\omega),\overline{\mV}_m\rangle\approx0\quad\mbox{if}\quad\mx\ne\mz_m.
\end{aligned}
\end{align}
Hence, we can introduce a filtering function, which is called subspace migration operated at $\omega$;
\begin{equation}\label{Subspacemigration}
  \mathbb{F}(\mx;\omega):=\left|\sum_{m=1}^{M}\left\langle\mathbf{W}(\mx;\omega),\mU_m\right\rangle\left\langle\mathbf{W}(\mx;\omega),\overline{\mV}_m\right\rangle\right|.
\end{equation}
Here, $\langle\mathbf{a},\mathbf{b}\rangle=\overline{\mathbf{a}}\cdot\mathbf{b}$. Then, based on the orthogonal property (\ref{orthonormalproperty}), we can find locations of $\mz_m\in\Gamma_m$ by finding $\mx$ such that $\mathbb{F}(\mx;\omega)\approx1$. A more detailed discussion can be found in \cite{AGKPS,JKHP,P1,P2,PL,PP}.

\section{Analysis of subspace migration and identification of exact locations}\label{sec3}
\subsection{Analysis of subspace migration with unknown frequency}
Although, subspace migration is fast, effective and stable imaging algorithm, the frequency $\omega$ must be known. However, it is very hard to identify locations of cracks via the map of $\mathbb{F}(\mx;\omega)$ when applied frequency $\omega$ is unknown. This fact is well-known but has been identified via the results of numerical simulations. Due to this reason, most of researches have been performed with \textit{a priori} information of $\omega$. In this section, we carefully explore the structure of subspace migration with unknown frequency and discuss its properties. For this purpose, we introduce a useful Lemma. This will play a key role in our exploration.

\begin{lem}[See \cite{KP}]\label{TheoremBessel}
  Suppose that $\set{\vt_n:n=1,2,\cdots,N}$ spans unit circle $\mathbb{S}^1$. Then, following relation holds for sufficiently large $N$, and $\vt,\mx\in\mathbb{R}^2$.
  \[\frac{1}{N}\sum_{n=1}^{N}\exp(i\omega\vt_n\cdot\mx)=\frac{1}{2\pi}\int_{\mathbb{S}^1}\exp(i\omega\vt\cdot\mx)d\vt=J_0(\omega|\mx|),\]
   where $J_0$ denotes the Bessel function of order zero of the first kind.
\end{lem}

Since, $\omega$ is unknown, we cannot use $\mW(\mx;\omega)$ of (\ref{VecW}). So, let us choose a value $\hat{\omega}$ and apply $\mW(\mx;\hat{\omega})$ to (\ref{Subspacemigration}) such that
\[\mathbb{F}(\mx;\hat{\omega}):=\left|\sum_{m=1}^{M}\left\langle\mathbf{W}(\mx;\hat{\omega}),\mU_m\right\rangle\left\langle\mathbf{W}(\mx;\hat{\omega}),\overline{\mV}_m\right\rangle\right|.\]
Then, we can obtain the following main result.
\begin{thm}\label{StructureofSubspaceMigration}
 For sufficiently large $N$, subspace migration imaging functional $\mathbb{F}(\mx;\hat{\omega})$ can be represented as follows:
  \[\mathbb{F}(\mx;\hat{\omega})\approx\sum_{m=1}^{M}J_0(\hat{\omega}|\mx-\hat{\mz}_m|)^{2}\quad\mbox{with}\quad\hat{\mz}_m=\frac{\omega}{\hat{\omega}}\mz_m.\]
\end{thm}
\begin{proof}
  Since the incident and observation direction configurations are same, we set $\triangle\vt_j:=|\vt_j-\vt_{j-1}|$ for $j=2,3,\cdots,N,$ and $\triangle\vt_1:=|\vt_1-\vt_N|$. Then, applying (\ref{SingularVectors}) and Lemma \ref{TheoremBessel} yields
  \begin{align*}
    \mathbb{F}(\mx;\hat{\omega})&=\left|\sum_{m=1}^{M}\left\langle\mathbf{W}(\mx;\hat{\omega}),\mU_m\right\rangle\left\langle\mathbf{W}(\mx;\hat{\omega}),\overline{\mV}_m\right\rangle\right|\approx\left|\sum_{m=1}^{M}\left\langle\mathbf{W}(\mx;\hat{\omega}),\mathbf{W}(\mz_m;\omega)\right\rangle^2\right|\\
    &=\sum_{m=1}^{M}\left(\frac{1}{N}\sum_{n=1}^{N}\exp(i\vt_n\cdot(\hat{\omega}\mx-\omega\mz_m))\frac{\triangle\vt_n}{2\pi}\right)^2=\frac{1}{4\pi^2}\sum_{m=1}^{M}\left(\int_{\mathbb{S}^1}\exp(i\vt\cdot(\hat{\omega}\mx-\omega\mz_m))d\vt\right)^2\\
    &=\sum_{m=1}^{M}J_0(|\hat{\omega}\mx-\omega\mz_m|)^2=\sum_{m=1}^{M}J_0(\hat{\omega}|\mx-\hat{\mz}_m|)^{2},
  \end{align*}
  where $\hat{\mz}_m=\omega\mz_m/\hat{\omega}$. This completes the proof.
\end{proof}

Note that $J_0(x)$ has the maximum value $1$ at $x=0$. Hence, map of $\mathbb{F}(\mx;\hat{\omega})$ will plot magnitude $1$ at $\hat{\mz}_m$ instead of true location $\mz_m$. This is the reason why inexact location of cracks are extracted via subspace migration.

\subsection{Identification of exact location of cracks: generating small scatterer}
Based on the result in Theorem \ref{TheoremBessel}, the values $\omega\mz_m$ can be extracted via subspace migration but true location $\mz_m$ cannot be identified unless estimating the value of unknown frequency $\omega$. At this moment, we propose an algorithm to estimate the value of unknown frequency $\omega$.

The idea is very simple. Let us create a small scatterer at $\my\in\mathbb{R}^2\backslash\Gamma$. Then, we can extract the value of $\omega\my$ via subspace migration. Then, since we know the location of $\my$, we can obtain an estimated value of $\omega$. Hence, correspondingly, it will be possible to identify the location of $\mz_m$. But, one has a problem: there must be no overlap between $\my$ and $\mz_m$ for $m=1,2,\cdots,M$. If not, this method will be fail. However, we have no \textit{a priori} information of $\mz_m$ so, selecting location of $\my$ must be considered beforehand.

Fortunately, we know the value of $\omega\mz_m$. This means that although we don't know the location of $\mz_m$, we can find a line $\mathcal{L}_m$ such that $\omega\mz_m\in\mathcal{L}_m$ for $m=1,2,\cdots,M$. Hence, we can estimate $\omega$ and correspondingly $\mz_m$ by selecting $\my\notin\mathcal{L}_m$.

\section{Results of numerical simulation}\label{sec4}
In this section, we present some results of numerical simulation for supporting Theorem \ref{StructureofSubspaceMigration}. For this purpose, $N=20$ different incident and observation directions are applied such that
\[\vt_n=\left[\cos\frac{2\pi n}{N},\sin\frac{2\pi n}{N}\right]^T.\]
As we mentioned in Section \ref{sec2}, applied angular frequency is of the form $\omega=2\pi/\lambda$ and $\lambda=0.4$ is applied. Hence, unknown frequency is $\omega\approx15.7080$.

Every far-field pattern data $u_\infty(\vv_j,\vt_l)$, $j,l=1,2,\cdots,N$ is generated by the second-kind Fredholm integral equation along the crack introduced in \cite[Chapter 4]{N} for avoiding the inverse crime. After obtaining the dataset, a $20$dB white Gaussian random noise is added to the unperturbed data. Three cracks $\Sigma_s$ with small length $\ell=0.05$ are chosen for numerical simulations such that
\begin{align*}
  \Sigma_1&=\set{[t-0.6,-0.2]^T:-\ell\leq t\leq\ell}\\
  \Sigma_2&=\set{\mathcal{R}_{\pi/4}[t+0.4,t+0.35]^T:-\ell\leq t\leq\ell}\\
  \Sigma_3&=\set{\mathcal{R}_{7\pi/6}[t+0.25,t-0.6]^T:-\ell\leq t\leq\ell}.
\end{align*}
Here, $\mathcal{R}_{\theta}$ denotes the rotation by $\theta$.

\begin{figure}[!ht]
\begin{center}
\includegraphics[width=0.45\textwidth]{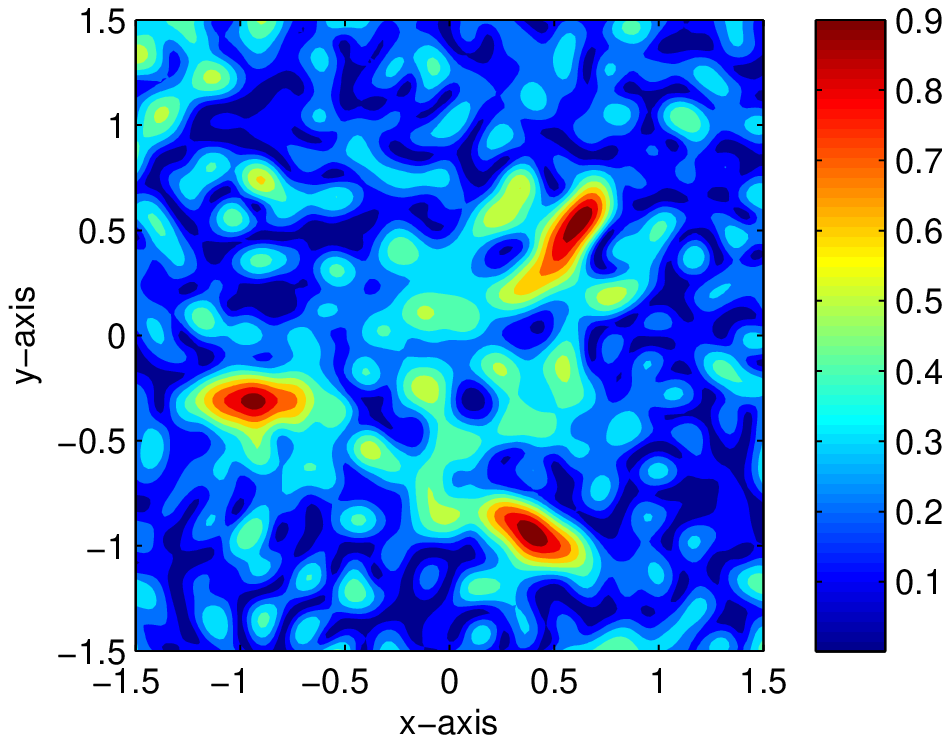}
\includegraphics[width=0.45\textwidth]{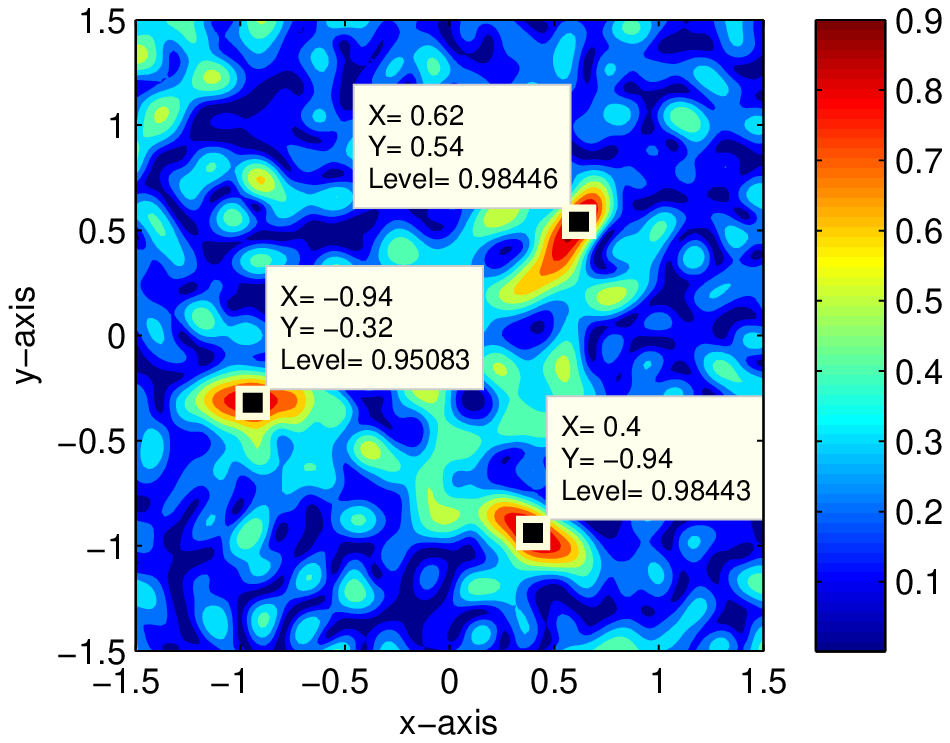}
\caption{\label{Figure1}Maps of $\mathbb{F}(\mx;10)$ with identified locations.}
\end{center}
\end{figure}

\begin{figure}[!ht]
\begin{center}
\includegraphics[width=0.45\textwidth]{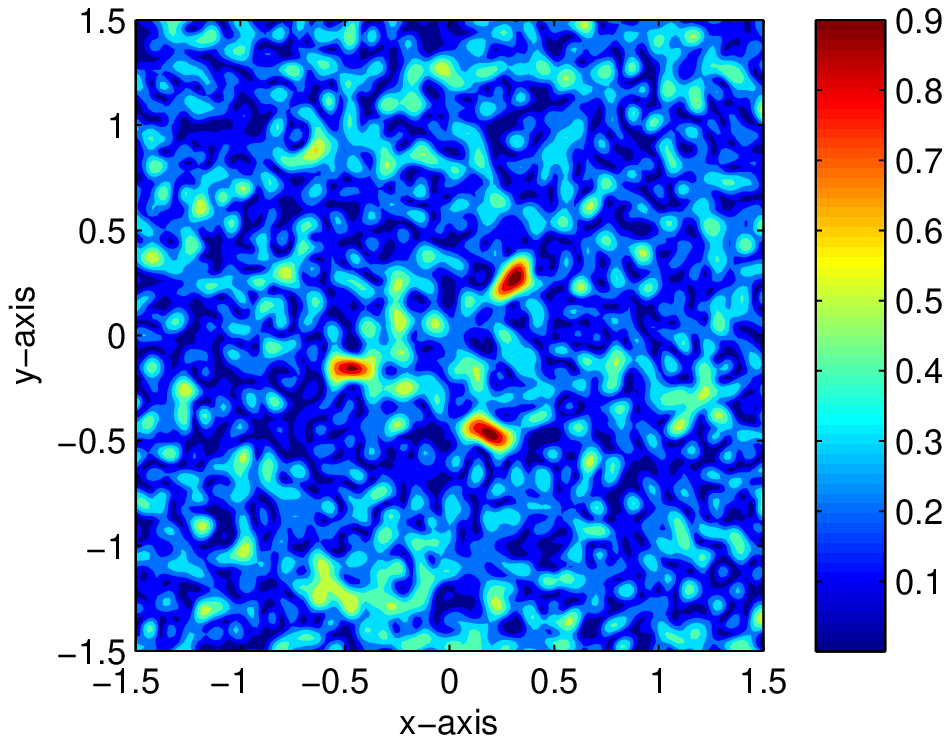}
\includegraphics[width=0.45\textwidth]{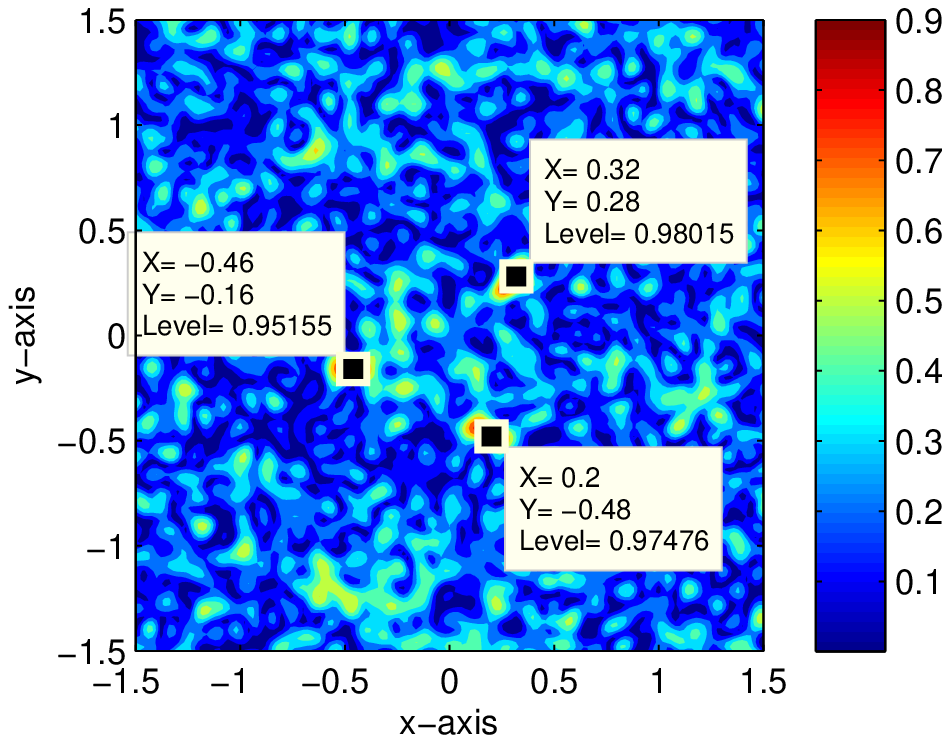}
\caption{\label{Figure2}Maps of $\mathbb{F}(\mx;20)$ with identified locations.}
\end{center}
\end{figure}

Now, let us identify location of $\mz_m$. Based on the results in Figures \ref{Figure1} and \ref{Figure2}, the values $\omega\mz_m$ can be identified so that three lines $\mathcal{L}_m$ are also. For example, based on the result in Figure \ref{Figure2}, three lines are given by (see Figure \ref{FigureLine} also)
\begin{align*}
  \mathcal{L}_1&:y=\frac{0.28}{0.32}x=0.8750x\quad(x>0),\\
  \mathcal{L}_2&:y=\frac{-0.16}{-0.46}x=0.3478x\quad(x<0),\\
  \mathcal{L}_3&:y=\frac{-0.48}{\phantom{-}0.20}x=-2.400x\quad(x>0).
\end{align*}

\begin{figure}[!ht]
\begin{center}
\includegraphics[width=0.45\textwidth]{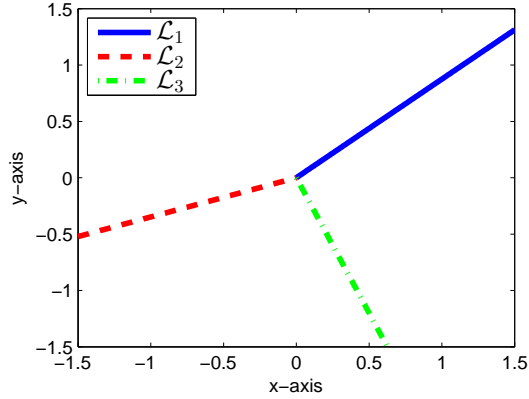}
\caption{\label{FigureLine}Three lines $\mathcal{L}_m$. $\mz_m$ is located somewhere on the $\mathcal{L}_m$.}
\end{center}
\end{figure}

Hence, we can create small scatterer at $\my\ne\mz_m\in\mathcal{L}_m$. Note that $\my$ must located far away from $\mz_m$ so, in this example, we create a small scatterer at $\my=[y_1,y_2]=[1.5,0]^T$. With this, by regarding map of $\mathbb{F}(\mx;20)$, we can obtain the value of $\hat{\my}=[\hat{y}_1,\hat{y}_2]^T=\omega\my/\hat{\omega}=[1.18,0]^T$. Hence, estimated angular frequency is
\[\omega\approx\frac{\hat{\omega}\hat{y}_1}{y_1}=\frac{1.18\times20}{1.5}=15.7333=\frac{2\pi}{0.4098}.\]
This value is almost same as the true frequency. Hence, almost exact locations $\mz_m$ can be identified via the map of $\mathbb{F}(\mx;15.7333)$, refer to Figure \ref{Figure4}.

\begin{figure}[!ht]
\begin{center}
\includegraphics[width=0.45\textwidth]{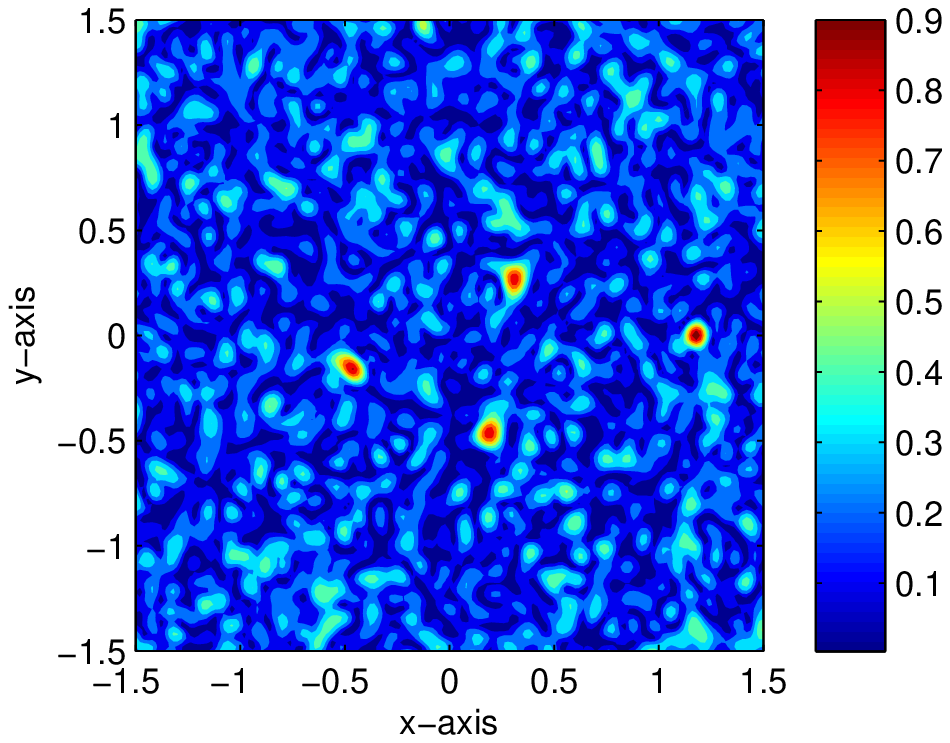}
\includegraphics[width=0.45\textwidth]{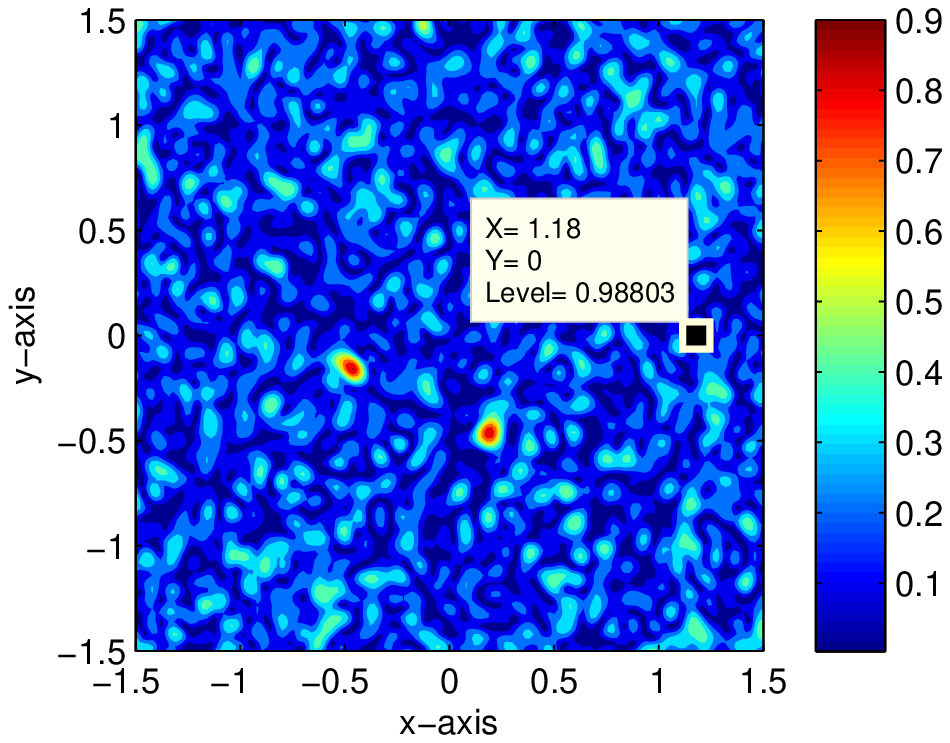}
\caption{\label{Figure3}Maps of $\mathbb{F}(\mx;20)$ and identified location of created scatterer.}
\end{center}
\end{figure}

\begin{figure}[!ht]
\begin{center}
\includegraphics[width=0.45\textwidth]{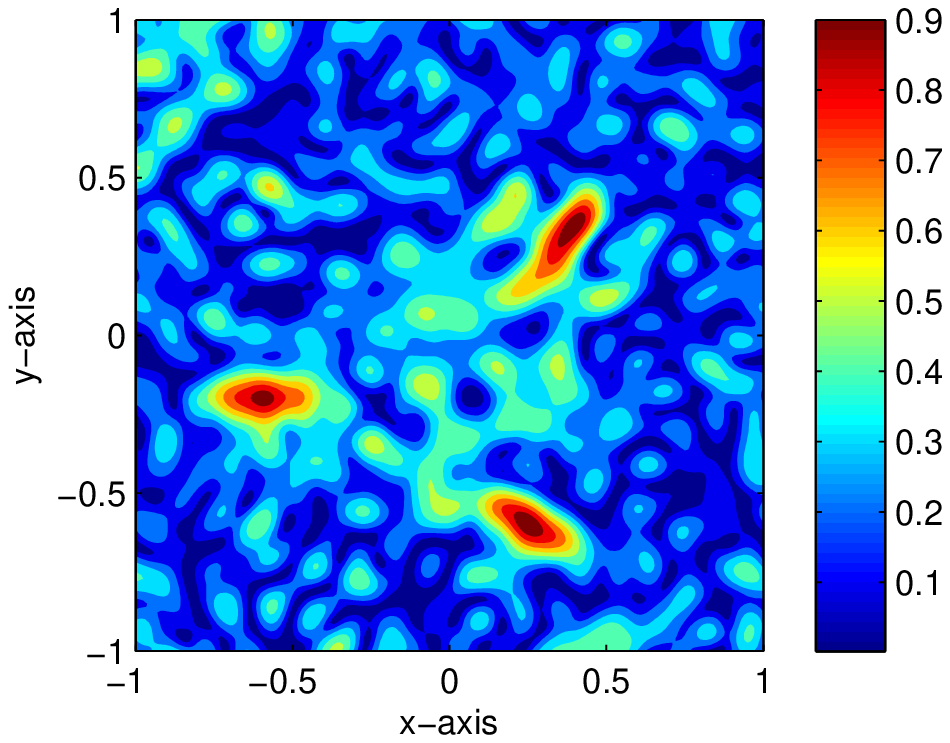}
\includegraphics[width=0.45\textwidth]{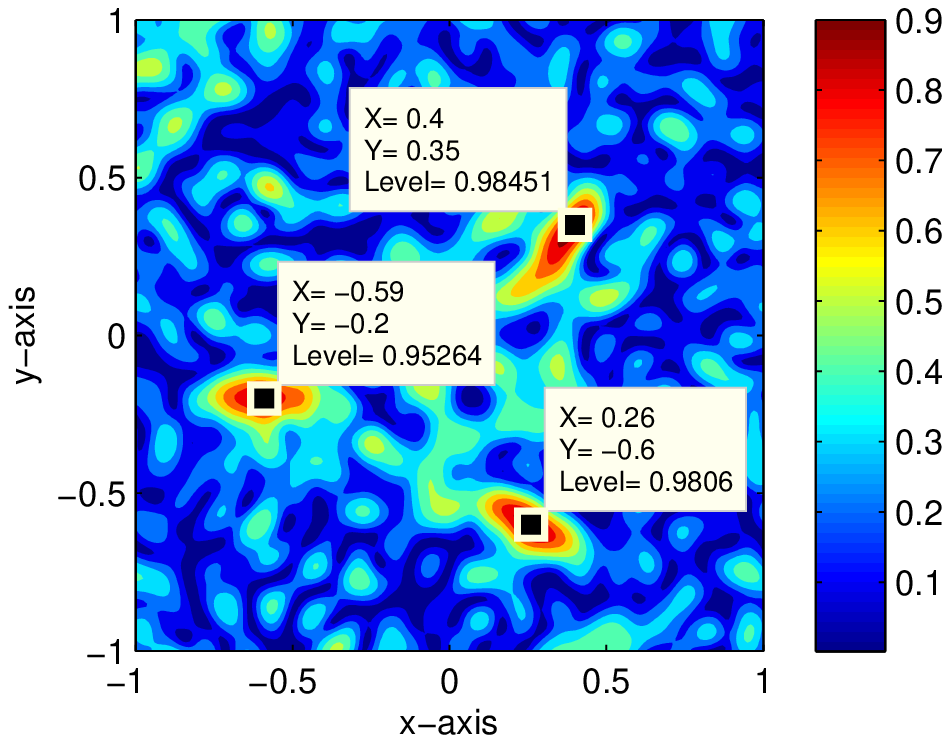}
\caption{\label{Figure4}Maps of $\mathbb{F}(\mx;15.7333)$ with identified locations.}
\end{center}
\end{figure}

\section{Conclusion}\label{sec5}
The structure of subspace migration functional for finding location of small perfectly conducting cracks is investigated when applied frequency is unknown. Based on its relationship with Bessel function of order zero of the first kind, we have confirmed the reason of ineffectiveness of subspace migration with unknown frequency information. Fortunately, based on the tendency of inaccurate result, we developed a simple algorithm for finding exact location of cracks by creating a small scatterer.

The main subject of this paper is imaging of small cracks in the two-dimensional problems. Extension to the arbitrary shaped arc-like cracks will be the forthcoming work. Moreover, development of exact location search algorithm for half-space problem \cite{AIL,P2,PL} or limited-view problem \cite{KP} will be an interesting research subject. Finally, we expect the methodology in this paper could be extended to the three-dimensional problem.

\end{document}